\documentclass[11pt,a4paper]{article}
\usepackage{amsfonts,amsgen,amsmath,amstext,amsbsy,amsopn,amsthm,amsfonts,amssymb,amscd}
\usepackage{epsf,epsfig}
\usepackage{float}
\usepackage{ebezier,eepic}
\usepackage{color}
\usepackage{tikz}
\usepackage{multirow}
\usepackage{mathrsfs}
\usepackage{graphicx}
\usepackage{subfigure}
\newtheorem{thm}{Theorem}[section]

\newtheorem{lem}[thm]{Lemma}
\newtheorem{false statement}{False statement}
\newtheorem{cor}[thm]{Corollary}

\theoremstyle{definition}
\newtheorem{defn}[thm]{Definition}

\newtheorem{conj}[thm]{Conjecture}

\makeatletter \@addtoreset{equation}{section}

\def\hf{\mathcal{F}}

\def\ha{\mathcal{A}}
\def\hb{\mathcal{B}}

\def\hc{\mathcal{C}}

\def\hp{\mathcal{P}}
\begin{document}
\title{\bf\Large On the sum of sizes of overlapping families}
\date{}
\author{Peter Frankl$^1$, Jian Wang$^2$\\[10pt]
$^{1}$R\'{e}nyi Institute, Budapest, Hungary\\[6pt]
$^{2}$Department of Mathematics\\
Taiyuan University of Technology\\
Taiyuan 030024, P. R. China\\[6pt]
E-mail:  $^1$peter.frankl@gmail.com, $^2$wangjian01@tyut.edu.cn
}

\maketitle

\begin{abstract}
Let $\ha_1,\ldots,\ha_m$ be families of $k$-subsets of an $n$-set. Suppose that one cannot choose pairwise disjoint edges from $s+1$ distinct families. Subject to this condition we investigate the maximum of $|\ha_1|+\ldots+|\ha_m|$. Note that the subcase $m=s+1$, $\ha_1=\ldots=\ha_m$ is the Erd\H{o}s Matching Conjecture, one of the most important open problems in extremal set theory. We provide some upper bounds, a general conjecture and its solution for the range $n\geq 4k^2s$.
\end{abstract}

\section{Introduction}

Let $[n]=\{1,2,\ldots,n\}$ and $\ha_1,\ldots,\ha_m\subset 2^{[n]}$. Define the rainbow matching number $\nu(\ha_1,\ldots,\ha_m)$ as the maximum integer $s$ such that there exist pairwise disjoint sets $A_{i_1}\in \ha_{i_1}, A_{i_2}\in \ha_{i_2},\ldots, A_{i_s}\in \ha_{i_s}$ with $1\leq i_1<i_2<\ldots<i_s\leq m$. Define $f(n,k,m,s)$ as the maximum of $|\ha_1|+\ldots+|\ha_m|$ over all $\ha_1,\ldots,\ha_m\subset \binom{[n]}{k}$ with $\nu(\ha_1,\ldots,\ha_m)\leq s$.

\vspace{6pt}
{\noindent \bf Hilton Theorem.} (\cite{H77})
If $n\geq 2k$, $s=1$, $m\geq 1$, then
\begin{align}\label{hilton}
f(n,k,m,1) =\max\left\{\binom{n}{k},m\binom{n-1}{k-1}\right\}.
\end{align}

\vspace{6pt}
{\noindent \bf Observation 1.} Let $\ha,\hb\subset 2^{[n]}$ and let $\hp(\ha,\hb)$ denote the collection of all unordered pairs $(A,B)$ such that $A\in \ha,B\in \hb$. Then
\[
\hp(\ha,\hb) \supset \hp(\ha \cap \hb,\ha\cup \hb).
\]
In view of this observation the rainbow matching number cannot increase if we replace $\ha_i,\ha_j$ by $\ha_i\cap \ha_j$, $\ha_i\cup \ha_j$. Repeating this operation will eventually produce a nested family, i.e., $\ha_1\subset \ldots \subset \ha_m$. Since $|\ha|+|\hb|=|\ha\cap \hb|+|\ha\cup \hb|$, in proving an upper bound for $f(n,k,m,s)$ we can always assume that $\ha_1,\ldots,\ha_m$ are nested.

\vspace{6pt}
{\noindent \bf Observation 2.} If $\ha_1,\ldots,\ha_m\subset 2^{[n]}$ are nested and  $\nu(\ha_1,\ldots,\ha_m)\leq s$, then $\ha_{m-s},\ha_{m-s},$\\$\ldots,\ha_{m-s},\ha_{m-s+1},\ldots,\ha_m$ have also rainbow matching number at most $s$.

\begin{cor}
Let $n,k,m, s$ be integers with $n\geq(k+1)s$ and $m\geq s+1$. Then
\begin{align*}
f(n,k,m,s) =&\max\Big\{(m-s)|\hb_0|+|\hb_1|+\ldots+|\hb_s|\colon \hb_0\subset\hb_1\subset \ldots\subset \hb_s\subset \binom{[n]}{k},\\
 &\nu(\hb_0,\ldots, \hb_s)\leq s\Big\}.
\end{align*}
\end{cor}

\begin{defn}
A sequence of families $\hb_0,\hb_1,\ldots,\hb_s\subset 2^{[n]}$ is called overlapping if $\nu(\hb_0,\hb_1,\ldots,\hb_s)\leq s$.
\end{defn}

Let $p_0\geq p_1\geq \ldots\geq p_s$ be positive constants. Set $\vec{p}=(p_0,\ldots,p_s)$. Fix $n,k$ where $n\geq (s+1)k$. Define
\[
f_{\vec{p}}(n,k,s) =\max\left\{\sum_{i=0}^s p_i |\hb_i|\colon \hb_0\subset \ldots\subset\hb_s \subset \binom{[n]}{k} \mbox{ are overlapping}\right\},
\]
$f(n,k,m,s)$ corresponds to the case $p_0=m-s,p_1=\ldots=p_s=1$.

By applying Katona's cyclic permutation method \cite{K72}, we prove an upper bound on $f_{\vec{p}}(n,k,s)$ for $\vec{p}=(p,1,\ldots,1)$.

\begin{thm}\label{main-1}
Let $p$ be a positive integer and $\vec{p}=(p,1,\ldots,1)$. For $n\geq (s+1)k$,
\[
f_{\vec{p}}(n,k,s) \leq \max\left\{s\binom{n}{k},(p+s)s\binom{n-1}{k-1}\right\}.
\]
\end{thm}

Note that the construction $\hb_0=\emptyset, \hb_1=\ldots =\hb_s=\binom{[n]}{k}$ shows that $f_{\vec{p}}(n,k,s)=s\binom{n}{k}$ for $n\geq (p+s)k$. It should be mentioned that the case $n\geq (p+s)k$ has already been proved in \cite{F13} and was used in the proof of the Erd\H{o}s matching conjecture in the current best range by the first author.

Using the shifting technique (cf. \cite{F87}), we determine  $f_{\vec{p}}(n,k,s)$ for $n\geq 4k^2s$ and $\vec{p}=(p,1,\ldots,1)$.
\begin{thm}\label{main-2}
Let $p$ be a positive integer and $\vec{p}=(p,1,\ldots,1)$. For $n\geq 4k^2s$,
\[
f_{\vec{p}}(n,k,s) = \max\left\{s\binom{n}{k},(p+s)\left(\binom{n}{k}-\binom{n-s}{k}\right) \right\}.
\]
\end{thm}

Let
\[
\mathcal{E}(n,k,s)=\left\{E\in{[n]\choose k}\colon E\cap [s]\neq \emptyset\right\}.
\]
The constructions $\hb_0=\emptyset,\  \hb_{1}=\ldots =\hb_s=\binom{[n]}{k}$ and $\hb_0=\hb_1=\ldots =\hb_s=\mathcal{E}(n,k,s)$ provide the lower bound in Theorem \ref{main-2}.

For general vectors $\vec{p}$, the following two results are determined.

\begin{thm}\label{main-3}
For arbitrary $p_0\geq p_1\geq \ldots\geq p_s>0$,
\[
f_{\vec{p}}((s+1)k,k,s) = (p_0+p_1+\ldots+p_s) \binom{(s+1)k-1}{k}.
\]
\end{thm}

\begin{thm}\label{main-4}
For arbitrary $p_0\geq p_1\geq \ldots \geq p_s>0$. Set
\[
 d_{\vec{p}} =\max_{i\geq 1} \frac{i(p_0+\ldots +p_s)}{p_1+\ldots+p_i}.
\]
If $n\geq \max\{(s+1)k, \lceil d_{\vec{p}}\rceil k\}$, then
\[
f_{\vec{p}}(n,k,s) =(p_1+\ldots+p_s) \binom{n}{k}.
\]
\end{thm}

\section{The proof of Theorem \ref{main-1}}

In this section, we prove an upper bound for $f_{\vec{p}}(n,k,s)$ in the case $\vec{p}=(p,1,\ldots,1)$. First we recall some definitions. Let $\sigma=(x_0,\ldots,x_{n-1})$ be a cyclic permutation on $[n]$, $1\leq k<n$ an integer. Define the {\it arc} $A_k(x_i)$ on $\sigma$ as the $k$-set $(x_i,x_{i+1},\ldots,x_{i+k-1})$ where the computation is modulo $n$.  Let
 \[
 \hc(\sigma, k) =\left\{A_k(x_i)\colon i=0,1,\ldots,n-1\right\}.
 \]
 The cyclic permutation method was invented by Katona \cite{K72} to give a simple proof of the Erd\H{o}s-Ko-Rado theorem \cite{ekr}. Various applications of the cyclic permutation method in extremal set theory can be found in a recent paper \cite{F20}. By the cyclic permutation method, Theorem \ref{main-1} is implied by the following lemma.

\begin{lem}
 Let $n,k,p,s$ be positive integers. Let $\sigma$ be a cyclic permutation on $[n]$ and $\hb_0\subset \hb_1\subset \ldots\subset\hb_{s}\subset \hc(\sigma, k)$ be nested overlapping families. If  $n\geq (k+1)s$,  then
\[
p|\hb_0|+|\hb_1|+\ldots +|\hb_s|\leq \max \left\{ns,(p+s)ks\right\}.
\]
\end{lem}

\begin{proof}
 Without loss of generality, we assume that $\sigma=(1,2,\ldots,n)$. Let $t=\lfloor n/k\rfloor$  and $r= n-tk$.  Let
 \[
 X=\hc(\sigma, k) \mbox{ and } Y=\{\hb_{-1},\hb_{-2},\ldots,\hb_{-p},\hb_1, \ldots,\hb_{s}\}
 \]
where $\hb_{-1},\hb_{-2},\ldots,\hb_{-p}$ are identical copies of $\hb_0$.
We construct a bipartite graph $G$ with partite sets $X$, $Y$ and put an edge if $A_k(i)\in \hb_j$. Then it suffices to show that  $e(G)\leq \max\{sn,(p+s)ks\}$. We distinguish two cases as follows:

{\noindent\bf Case 1.} $p+s \leq t$. Choose an element $i$ from $[n]$ uniformly at random. Set
\[
M=\left\{A_k(i),A_k(i+k),\ldots,A_k(i+(t-1)k)\right\}
\]
and note that $M$ consists  of $t$ pairwise disjoint $k$-arcs. Clearly, we have $\nu(G[M,Y])\leq s$. By the K\"{o}nig-Hall Theorem there is a vertex cover  $S$ of size $s$ in  $G[M,Y]$. Set $x=|S\cap Y|$. It follows that
\begin{align}\label{eq-1}
e(M,Y) \leq (s-x)(p+s) +xt-x(s-x)\leq \max\{st, s(p+s)\} = st.
\end{align}
Since $i$ is chosen randomly, the expectation of $e(M,Y)$ equals
\begin{align}\label{eq-2}
\sum_{J\in X} \deg(J) Pr(J\in M) &= \sum_{J\in X} \deg(J) \cdot\sum_{j=0}^{t-1} Pr(J=A_k(i+jk))\nonumber \\[5pt]
&= \sum_{J\in X} \deg(J)\cdot \frac{t}{n}\nonumber \\[5pt]
&=\frac{t}{n}\cdot e(X,Y).
\end{align}
Combining \eqref{eq-1} and \eqref{eq-2}, we conclude that $e(G)=e(X,Y)\leq sn$.

{\noindent\bf Case 2.} $p+s \geq t+1$. Rearrange the arcs in $X$ as $J_1,J_2,\ldots,J_n$ so that
\[
\deg_G(J_1)\leq \deg_G(J_2)\leq \ldots \leq \deg_G(J_n).
\]
Let $U=\{J_{r+1},J_{r+2},\ldots,J_n\}$ and $1\leq y_1<y_2<\ldots<y_{kt}\leq n$ in this order be the heads of the arcs in $U$. For $1\leq j \leq k$,  define
\[
M_j=\left\{A_k(y_j),A_k(y_{j+k}),\ldots,A_k(y_{j+(t-1)k})\right\}.
\]
It is easy to check that  $M_j$ consists of $t$ pairwise disjoint $k$-arcs, implying that $\nu(G[M_j,Y])\leq s$. Clearly we have $\deg(J_r)\leq s$. Otherwise, each vertex in $M_j$ has degree at least $s+1$ in $G[M_j,Y]$, contradicting the fact that $\nu(G[M_j,Y])\leq s$.

Since $\nu(G[M_j,Y])\leq s$, by the K\"{o}nig-Hall Theorem there is a vertex cover  $S$ of size $s$ in  $G[M_j,Y]$. Set $x=|S\cap Y|$. It follows that for $1\leq j\leq k$
\begin{align}\label{neq-7}
e(M_j,Y) \leq (s-x)(p+s) +xt-x(s-x)\leq \max\{st, s(p+s)\} = s(p+s).
\end{align}

If $\deg(J_r)=0$, then by  \eqref{neq-7} we have
\[
e(X,Y) =\sum_{j=1}^k e(M_j,Y)\leq ks(p+s).
\]
If $\deg(J_r)\geq 1$, then each arc in $U$ has degree at least one.  It follows that $1\leq x\leq s$. Then we have for $1\leq j\leq k$
\[
e(M_j,Y) \leq (s-x)(p+s) +xt-x(s-x)\leq \max\{st, (s-1)(p+s) + t-s+1\}.
\]
Since $t\leq p+s-1$, it follows that
\[
e(M_j,Y) \leq s(p+s-1).
\]
Note that $\deg(J_r)\leq s$ and $r<k$. Hence
\[
e(X,Y) =\sum_{j=1}^k e(M_j,Y)+ rs\leq k(p+s-1)s+rs <  k(p+s)s.
\]
Thus the lemma holds.
\end{proof}

\section{The exact value of $f_{\vec{p}}(n,k,s)$ for large $n$}
In this section, we determine  $f_{\vec{p}}(n,k,s)$ for $n\geq 4k^2s$ and $\vec{p}=(p,1,\ldots,1)$. We first show that the result is implied by the following lemma.

\begin{lem}\label{main-22}
Let $n,p,s,i$ be positive integers  and $\vec{p}=(p,1,\ldots,1)$. Define
\[
g_k(n,p,s,i) = (p+i)\left(\binom{n}{k}-\binom{n-i}{k}\right)+(s-i)\binom{n}{k}=(p+s)\binom{n}{k}-(p+i)\binom{n-i}{k}.
\]
If $n\geq 4k^2s$, then
\begin{align}\label{eq-3}
f_{\vec{p}}(n,k,s) = \max_{0\leq i\leq s} g_k(n,p,s,i).
\end{align}
\end{lem}

\begin{proof}[Proof of Theorem \ref{main-2}.]
By Lemma \ref{main-22}, \eqref{eq-3} holds. Set $h(x) = g_k(n,p,s,x)$. Computing its derivative, we have
\[
h'(x) = \left((p+x)\sum_{j=0}^{k-1} \frac{1}{n-j-x} -1 \right)\binom{n-x}{k}.
\]
Set
\[
u(x) = (p+x)\sum_{j=0}^{k-1} \frac{1}{n-j-x} -1.
\]
It is easy to see that $u(x)$ is an increasing function in the range $[-p, n-k]$ with
\[
u(-p)=-1<0 \mbox{ and } u(n-k) =(p+n-k) \left(1+\frac{1}{2}+\ldots+\frac{1}{k}\right)-1>0.
\]
Let $x_0$ be the unique zero of $u(x)$ in $[-p,n-k]$. Then $h(x)$ is decreasing in $[-p,x_0]$ and increasing in $[x_0, n-k]$. Consequently, the maximum of $h(x)$ in the interval $[0,s]$ is attained for $x=0$ or $x=s$. Thus the theorem follows.
\end{proof}

{\noindent \bf Remark.} Recall that $f_{\vec{p}}(n,k,s)=s\binom{n}{k}$ for $p\leq \frac{n}{k}-s$. If $p\geq \frac{n}{k}$, then
\[
u(0) =p\sum_{j=0}^{k-1} \frac{1}{n-j} -1 \geq p \cdot\frac{k}{n}-1\geq 0.
\]
Thus for $n\geq 4k^2s$ and $p\geq \frac{n}{k}$ we have
\[
f_{\vec{p}}(n,k,s) =(p+s)\left(\binom{n}{k}-\binom{n-s}{k}\right).
\]

Now we turn to the proof of Lemma \ref{main-22}. In the proof, we need the following simple inequalities, which were already used  in \cite{BDE}.

\begin{lem}[\cite{BDE}]
\begin{align}
&l\binom{m-1}{s-1}\geq \binom{m}{s}-\binom{m-l}{s}\geq l\binom{m-l}{s-1},\label{neq-used1}\\[5pt]
&\binom{m-l}{s}/\binom{m}{s}\geq \left(1-\frac{l}{m-s}\right)^s\geq 1-\frac{sl}{m-s}.\label{neq-used2}
\end{align}
\end{lem}

We also need a general upper bound on the maximum number of edges in hypergraphs with given matching number.

\begin{lem}[\cite{F87}]\label{gbEMC}
Suppose that $\hf\subset \binom{[n]}{k}$, $n\geq k(s+1)$ and $\nu(\hf)\leq s$. Then
\[
|\hf| \leq s\binom{n-1}{k-1}.
\]
\end{lem}

In the proof of Lemma \ref{main-22} we also need some facts concerning shifting, an important operation invented by Erd\H{o}s, Ko and Rado \cite{ekr}. For proofs and details cf. \cite{F87}.

Shifting is an operation on families that maintains the size of the sets and of the families. It does not increase the size of the (rainbow) matching number. Repeated application of shifting permits to consider shifted families, a notion that we define below.

Let $(x_1,\ldots,x_k)$ denote the set $\{x_1,\ldots,x_k\}$ where we know that $x_1<\ldots<x_k$. The shifting partial order $\prec$ is defined as follows: $(x_1,\ldots,x_k)\prec (y_1,\ldots,y_k)$ iff $x_i\leq y_i$ for all $1\leq i\leq k$.

\begin{defn}
A family $\hf\subset \binom{[n]}{k}$ is called shifted if $G\prec F\in \hf$ always implies $G\in \hf$.
\end{defn}

By the above consideration, in the proof of Lemma \ref{main-22} we can assume that all families are shifted.

\begin{proof}[Proof of Lemma \ref{main-22}.]
We apply induction on $s$. The case $s=1$ was proved in \cite{F20}. Without loss of generality, we may assume that $\ha_0\subset \ha_1\subset \ldots \subset\ha_s\subset\binom{[n]}{k}$ are $s+1$ overlapping families with $p|\ha_0|+\sum_{i=1}^s |\ha_i|$ maximum.
If $\ha_s=\binom{[n]}{k}$, then
$\ha_0, \ha_1, \ldots, \ha_{s-1}$ are overlapping. By the induction hypothesis, we have
\[
p|\ha_0|+\cdots+|\ha_{s}|\leq \max_{0\leq i\leq s-1} g_k(n,p,s-1,i)+|\ha_{s}| \leq \max_{0\leq i\leq s-1} g_k(n,p,s,i).
\]
Thus, we may assume that $\ha_{s}\subsetneq \binom{[n]}{k}$.

First we show that the lemma holds if $(1,ks+2,\ldots,ks+k)\in \ha_0$. Let
\[
\hb_i =\ha_i\cap \binom{[2,n]}{k}
\]
for $i=1,\ldots,s$.
Since $\ha_i$ is shifted, $\hb_i$ is also shifted. If $\hb_1,\ldots,\hb_s$ span a rainbow matching, then by shiftedness they span a rainbow matching on $[2,ks+1]$, which together with $(1,ks+2,\ldots,ks+k)\in \ha_0$ leads to a rainbow matching in $\{\ha_0,\ha_1,\ldots,\ha_s\}$, a contradiction. Thus, $\hb_1,\ldots,\hb_s$ are overlapping. Applying the induction hypothesis to the $s$ families $\hb_1,\ldots,\hb_s$ on $[2,n]$, we have
\[
(p+1)|\hb_1|+\ldots +|\hb_s| \leq \max_{0\leq i\leq s-1} g_k(n-1,p+1,s-1,i)
\]
Assuming that $p|\ha_0|+\sum_{i=1}^s |\ha_i|$ is maximal, we have
\[
|\ha_i|=|\hb_i|+\binom{n-1}{k-1}.
\]
Therefore,
\begin{align*}
p|\ha_0|+\cdots+|\ha_{s}|&= (p+1)|\hb_1|+\ldots +|\hb_s| +(p+s)\binom{n-1}{k-1}\\[5pt]
&\leq \max_{0\leq i\leq s-1} g_k(n-1,p+1,s-1,i)+(p+s)\binom{n-1}{k-1}\\[5pt]
&\leq \max_{1\leq i\leq s} g_k(n,p,s,i)
\end{align*}
and the lemma follows. Thus, we are left with the case $(1,ks+2,\ldots,ks+k)\notin \ha_0$.

{\noindent \bf Claim 1.} $(ks+1,ks+2,\ldots,ks+k)\notin \ha_s$.
\begin{proof}
Since $\ha_{s}\subsetneq \binom{[n]}{k}$ and $p|\ha_0|+\sum_{i=1}^s |\ha_i|$ is maximal, there exists some $e\notin \ha_s$ such that $\ha_0, \ha_1, \ldots, \ha_{s-1},\ha_s+e$ are not overlapping. It follows that $\ha_0, \ha_1, \ldots, \ha_{s-1}$ is not overlapping. Since $\ha_0, \ha_1, \ldots, \ha_{s-1}$ are shifted, there is a rainbow matching in  $\ha_0, \ha_1, \ldots, \ha_{s-1}$ on $[ks]$. Thus, $(ks+1,ks+2,\ldots,ks+k)\notin \ha_s$.
\end{proof}

By Claim 1, $[ks]$ is a vertex cover set of $\ha_s$. It follows that
\begin{align}\label{neq-1}
|\ha_s|\leq \binom{n}{k} -\binom{n-ks}{k}.
\end{align}
For each $i\in [n]$, let
\[
\ha_0(i) =\left\{T\in \binom{[i+1,n]}{k-1}\colon T\cup \{i\}\in \ha_0\right\}.
\]

{\noindent \bf Claim 2.} $|\ha_0(s+1)|\leq s\binom{n-2}{k-2}$.
\begin{proof}
Suppose to the contrary that $|\ha_0(s+1)|> s\binom{n-2}{k-2}$. Then by Lemma \ref{gbEMC} there is a matching $\{T_1,T_2,\ldots,T_{s+1}\}$ in $\ha_0(s+1)$. By shiftedness, $M=\{\{1\}\cup T_1, \{2\}\cup T_2,\ldots, \{s+1\}\cup T_{s+1}\}$ is a matching of size $s+1$ in $\ha_0$. Since $\ha_0\subset \ha_1\subset \ldots \subset\ha_s$, $M$ is also a rainbow matching of size $s+1$. This contradicts the fact $\ha_0, \ha_1, \ldots, \ha_{s-1},\ha_s$ are overlapping.
\end{proof}

Since $(1,ks+2,\ldots,ks+k)\notin \ha_0$, each edge in $\ha_0$ contains at least two vertices in $[ks+1]$. It follows that
\[
|\ha_0(1)|\leq ks\binom{n-2}{k-2}.
\]
By Claim 2, we have
\begin{align}\label{neq-2}
|\ha_0| \leq \sum_{i=1}^{ks} |\ha_0(i)| \leq  s|\ha_0(1)|+(ks-s)|\ha_0(s+1)|\leq (2k-1)s^2\binom{n-2}{k-2}.
\end{align}
Combining \eqref{neq-1} and \eqref{neq-2}, we arrive at
\begin{align}\label{neq-3}
p|\ha_0|+\cdots+|\ha_{s}|\leq p(2k-1)s^2\binom{n-2}{k-2}+s\left(\binom{n}{k} -\binom{n-ks}{k}\right).
\end{align}
Now we distinguish two cases.

{\noindent\bf Case 1.} $p\leq 4(k-1)s$.

By \eqref{neq-used2} and $n\geq 2k^2s+k$, we have
\[
\binom{n-ks}{k}/\binom{n}{k} \geq 1-\frac{k^2s}{n-k} \geq \frac{1}{2}.
\]
Since $n\geq 4k^{2}s$, we have
\begin{align}\label{neq-4}
 p(2k-1)s^2\binom{n-2}{k-2} &\leq  4(k-1)s\cdot(2k-1)s^2\cdot \frac{k(k-1)}{n(n-1)}\binom{n}{k}\nonumber\\[5pt]
  &\leq \frac{8k^4s^2}{n^2} s \cdot 2\binom{n-ks}{k}\leq s\binom{n-ks}{k}.
\end{align}
Substituting \eqref{neq-4} into \eqref{neq-3}, we have
\begin{align*}
p|\ha_0|+\cdots+|\ha_{s}|\leq s \binom{n}{k} = g_k(n,p,s,0).
\end{align*}

{\noindent\bf Case 2.} $p\geq 4(k-1)s$.

By \eqref{neq-used2} and $n\geq 3ks+k$, we have
\[
\binom{n-s}{k-1}/\binom{n-1}{k-1} \geq 1-\frac{(k-1)(s-1)}{n-k} \geq \frac{2}{3}.
\]
Since $n\geq 4k^2s$, we have
\begin{align}\label{neq-5}
p(2k-1)s^2\binom{n-2}{k-2}&\leq p\frac{(2k-1)s^2(k-1)}{n-1} \binom{n-1}{k-1}\nonumber \\[5pt]
&\leq p\frac{(2k-1)(k-1)s}{n-1}s \cdot \frac{3}{2} \binom{n-s}{k-1}\nonumber\\[5pt]
&\leq p\frac{3k(k-1)s}{n-1}\left(\binom{n}{k}-\binom{n-s}{k}\right)\nonumber\\[5pt]
&\leq \frac{3}{4}p\left(\binom{n}{k}-\binom{n-s}{k}\right)
\end{align}
and
\begin{align}\label{neq-6}
s\left(\binom{n-s}{k}-\binom{n-ks}{k}\right) &\leq \frac{p}{4(k-1)} (k-1)s\binom{n-s-1}{k-1}\nonumber\\[5pt]
&\leq \frac{p}{4}\left(\binom{n}{k}-\binom{n-s}{k}\right).
\end{align}
Therefore, by \eqref{neq-5} and \eqref{neq-6} we have
\begin{align*}
& \ p|\ha_0|+\cdots+|\ha_{s}|-(p+s)\left(\binom{n}{k}-\binom{n-s}{k}\right)\\[5pt]
=& \ p(2k-1)s^2\binom{n-2}{k-2}+s\left(\binom{n-s}{k}-\binom{n-ks}{k}\right)-p\left(\binom{n}{k}-\binom{n-s}{k}\right) \\[5pt]
\leq & \ 0.
\end{align*}

Consequently, in both cases we have
\[
p|\ha_0|+\cdots+|\ha_{s}|\leq \max_{0\leq i\leq s} g_k(n,p,s,i),
\]
which completes the proof.
\end{proof}

\section{Proofs of Theorems \ref{main-3} and \ref{main-4}}

\begin{proof}[Proof of Theorem \ref{main-3}.]
Let $\mathcal{B}_0\subset \mathcal{B}_1\subset \cdots \subset \mathcal{B}_s\subset \binom{[n]}{k}$ be  overlapping families. Note that
\[\binom{(s+1)k-1}{k}=\frac{s}{s+1}\binom{(s+1)k}{k}.
\]
 Set $n=(s+1)k$. Let
\[
\hp = \{F_0,\ldots,F_s\}\subset \binom{[n]}{k}
\]
be an arbitrary partition. Consider the weighted bipartite graph $G$ where we have an edge $(F_i,\hb_j)$ iff $F_i\in \hb_j$. This edge gets weight $p_j$.

Applying the K\"{o}nig-Hall Theorem, we can find $s$ vertices covering all edges of the bipartite graph. Let by symmetry $F_0,\ldots,F_{l-1}$ be the vertices of the covering set chosen from $\hf$ and $\hb_{l+1},\ldots,\hb_s$ the remaining $s-l$ chosen from the families. (Here we used $\mathcal{B}_0\subset \mathcal{B}_1\subset \cdots \subset \mathcal{B}_s$. Nestedness guarantees that the neighborhood of $\hb_i$ is contained in the neighborhood of $\hb_j$ for $i<j$.)

Let us estimate the total weight of the edges in $G$. For $F_i$ the total weight is at most $\leq p_0+\ldots+p_s$. For $\mathcal{B}_j$ the total weight is at most $(s+1)p_j$. Note that $p_0\geq \ldots\geq p_s\geq 0$ implies
\[
p_{l+1}+\ldots+p_s \leq \frac{s-l}{s+1}(p_0+\ldots+p_s).
\]
Thus the total weight of the edges in $G$ is at most
\[
l(p_0+\ldots +p_s)+(s-l)(p_0+\ldots +p_s) =s(p_0+\ldots +p_s)
\]
with equality possible only if $l$ or $s-l$ is 0. Moreover, in the case of $l=0$,
\[
(p_1+\ldots +p_s)=\frac{s}{s+1}(p_0+\ldots +p_s)
\]
yields $(p_1+\ldots +p_s)=sp_0$ whence $p_0=p_1=\ldots=p_s$. Using uniform random choice for $\hp$ or letting $\hp$ run over a Baranyai partition system proves the theorem.
\end{proof}

{\noindent \bf Remark.} The above analysis for equality shows that unless $\vec{p} =(p_0,p_0,\ldots,p_0)$, the only way to have equality in Theorem \ref{main-4} is $\hb_0=\hb_1=\ldots=\hb_s$, consequently, $\hb_s =\binom{(s+1)k-1}{k}$ and $\hb_s$ satisfies $\nu(\hb_s)=s$. In view of \cite{F87}, then  $\hb_s=\binom{Y}{k}$ for some $|Y| = (s+1)k-1$.

\begin{proof}[Proof of Theorem \ref{main-4}.]
Let $\mathcal{B}_0\subset \mathcal{B}_1\subset \cdots \subset \mathcal{B}_s\subset \binom{[n]}{k}$ be  overlapping families. Let $ t = \lfloor n/k\rfloor$ and choose a random matching $F_1,F_2,\ldots,F_t$. Consider the weighted bipartite graph $G$ where we have an edge $(F_i,\mathcal{B}_j)$ iff $F_i\in \mathcal{B}_j$. This edge gets weight $p_j$.

Applying the K\"{o}nig-Hall Theorem we can find $s$ vertices covering all edges of the bipartite graph. Let $F_1,\ldots,F_l$ be the vertices of the covering set chosen from $\hf$ and $\mathcal{B}_{l+1}, \ldots, \mathcal{B}_s$ the remaining $s-l$ chosen from the families.

For $F_i$ the total weight is at most $\leq p_0+\ldots+p_s$. For $\mathcal{B}_j$ the total weight is at most $tp_j$. Thus, the total weight of the edges in $G$ is at most
\begin{align*}
l(p_0+\ldots+p_s)+t(p_{l+1}+\ldots +p_s)&= t(p_{1}+\ldots +p_s)- t(p_1+\ldots +p_l)+l(p_0+\ldots+p_s)\\[5pt]
&\leq t(p_{1}+\ldots +p_s)-  d_{\vec{p}}(p_1+\ldots +p_l)+l(p_0+\ldots+p_s)\\[5pt]
&\leq t(p_{1}+\ldots +p_s).
\end{align*}

Since the probability
\[
Pr(F_i\in \hb_j) =\frac{|\hb_j|}{\binom{n}{k}},
\]
the expected value of the  total weight of the edges in $G$ is $\sum_{j=0}^s tp_i \frac{|\hb_j|}{\binom{n}{k}}$. Thus, the theorem follows.
\end{proof}

\section{Concluding remarks}

The following conjecture would imply all our results in this paper.

\begin{conj}\label{conj-1}
Let $p$ be a positive integer and $\vec{p}=(p,1,\ldots,1)$. For $n\geq (s+1)k$,
\[
f_{\vec{p}}(n,k,s) = \max \left\{ s\binom{n}{k}, (p+s) \binom{(s+1)k-1}{k},(p+s)\left(\binom{n}{k}-\binom{n-s}{k}\right)\right\}.
\]
\end{conj}

Note that Conjecture \ref{conj-1} would imply the Erd\H{o}s matching conjecture (EMC for short) which is the subcase $\ha_0=\ldots=\ha_s$.

 Another generalisation of EMC was proposed by Aharoni and Howard.

\begin{conj}[\cite{AH}]\label{conj-2}
Suppose that $\hb_0,\ldots,\hb_s\subset\binom{[n]}{k}$ are overlapping, $n\geq k(s+1)$. Then
\begin{align}\label{eqn-1}
\min_{0\leq i\leq s} |\hb_i| \leq \max\left\{\binom{(s+1)k-1}{k},\binom{n}{k}-\binom{n-s}{k}\right\}.
\end{align}
\end{conj}

Note that the case $k=2$ is a direct consequence of the results in \cite{AF85} (cf. also \cite{F87}). Huang, Loh and Sudakov \cite{HLS12} proved \eqref{eqn-1} for $n>3k^2s$. Very recently, Gao, Lu, Ma and Yu \cite{GLMY} proved \eqref{eqn-1} for $k=3$ and $s>s_0$. Another recent achievement is due to Lu, Wang and Yu \cite{LWY20} where \eqref{eqn-1} is proved for general $k$ in the range $n\geq 2k(s+1)$, $s>s_0$. We should mention that the case $s=1$ of \eqref{eqn-1} is an easy consequence of the Kruskal-Katona Theorem (cf. \cite{Daykin74}). These results suggest that Conjecture \ref{conj-1} might be easier to attack in the case $s>s_0$ as well.

\end{document}